\documentclass[12pt]{amsart}

\usepackage{etex}
\usepackage[english]{babel}
\usepackage[cp1250]{inputenc}
\usepackage{amssymb,amsthm,amsfonts,amsmath}
\usepackage{mathrsfs}
\usepackage{verbatim}
\usepackage{enumitem}

\usepackage{url}
\usepackage{amsopn}
\usepackage{graphicx}
\usepackage[usenames, dvipsnames]{color}
\DeclareGraphicsExtensions{.pdf,.png,.jpg}

\usepackage{enumitem,titletoc}
\definecolor{darkblue}{rgb}{0.0, 0.0, 0.55}
\usepackage[colorlinks,linkcolor=BrickRed,citecolor=OliveGreen,urlcolor=darkblue,hypertexnames=true]{hyperref}
\usepackage[a4paper,margin=2.5cm]{geometry}
\def\CC{\mathbb{C}}
\def\RR{\mathbb{R}}
\def\NN{\mathbb{N}}
\def\ZZ{\mathbb{Z}}

\def\sl{\mathfrak{sl}}
\def\gl{\mathfrak{gl}}
\def\Sym{\text{Sym}}
\def\C{\mathcal{C}}
\def\R{\mathcal{R}}
\def\I{\mathcal{I}}
\def\rf{\mathrm{Ref}}
\def\loc{\mathrm{Loc}}
\def\span{\mathrm{span}}
\def\endo{\mathrm{End}}
\def\sl2{\mathfrak{sl}_2}
\def\sl{\mathfrak{sl}}
\def\tr{\mathrm{tr}}
\def\id{\mathrm{id}}
\def\ev{\mathrm{ev}}

\def\sym{\mathrm{Sym}}
\def\Lin{\mathrm{Lin}}

\usepackage[english]{babel}
\usepackage[cp1250]{inputenc}
\usepackage{amssymb,amsthm,amsfonts,amsmath}
\usepackage{mathrsfs}
\usepackage{verbatim}
\usepackage{url}
\usepackage{color}
\usepackage[usenames,dvipsnames]{xcolor}
\usepackage{amsopn}
\usepackage{enumitem}
\usepackage{multicol}
\usepackage{verbatim}

\newtheorem{proposition}{Proposition}

\newtheorem{thm}{Theorem}
\newtheorem{corollary}{Corollary}
\newtheorem{lemma}{Lemma}

\begin{document}
\title[A local-global principle for linear dependence]{A local-global principle for linear dependence in enveloping algebras of Lie algebras}

\author[J.\ Cimpri\v c]{Jaka Cimpri\v c${}^1$}
\address{Jaka Cimpri\v c, Faculty of Mathematics and Physics, University of Ljubljana, Slovenia}
\email{jaka.cimpric@fmf.uni-lj.si}
\thanks{${}^1$Supported by the Slovenian Research Agency grant P1-0222.}

\author[A. Zalar]{Alja\v z Zalar${}^2$}
\address{Alja\v z Zalar, Faculty of Computer and Information Science, University of Ljubljana, Slovenia}
\email{aljaz.zalar@fri.uni-lj.si} 
\thanks{${}^2$Supported by the Slovenian Research Agency grant P1-0288.}


\date{\today}

\maketitle

\begin{abstract}
For an associative algebra $A$ and a class $\C$ of representations of $A$ the following question (related to nullstellensatz) makes sense:
Characterize all tuples of elements $a_1,\ldots,a_n \in A$ such that vectors $\pi(a_1)v,\ldots,\pi(a_n)v$ are linearly dependent
for every $\pi \in \C$ and every $v$ in the representation space of $\pi$. We answer this question in the following cases:
\begin{enumerate}
\item $A=U(L)$ is the enveloping algebra of a finite-dimensional complex Lie algebra $L$ and $\C$ is the class of all finite-dimensional representations of $A$.
\item $A=U(\sl_2(\CC))$ and $\C$ is the class of all finite-dimensional irreducible representations of $A$.
\item $A=U(\sl_3(\CC))$ and $\C$ is the class of all finite-dimensional irreducible representations of $A$ with sufficiently high weights.
\end{enumerate}
In case (1) the answer is: tuples that are linearly dependent over $\CC$ while in cases (2) and (3) the answer is: tuples that are linearly dependent over the center of $A$. 
Similar results have been proved before for free algebras and Weyl algebras.
\end{abstract}


Let $A$ be a complex associative algebra and let $\C$ be a class of representations of $A$.  
We say that the elements $p_1,\ldots,p_k \in A$ are $\C$-\textit{locally linearly dependent} (abbreviated as $\C$-LLD) if for every
representation $\pi:A\to \endo(V_\pi)$ in $\C$ we have that $\pi(p_1),\ldots,\pi(p_k)$ are linearly dependent. 
We say that elements $p_1,\ldots,p_k \in A$ are $\C$-\textit{locally directionally linearly dependent} (abbreviated as $\C$-LDLD) if for every
representation $\pi:A\to \endo(V_\pi)$ in $\C$ and every vector $v \in V_\pi$ we have that $\pi(p_1)v,\ldots,\pi(p_k)v$ are linearly dependent. 
Clearly, linear dependence implies $\C$-LLD which implies $\C$-LDLD. The opposite implications are false in general. 
The motivation for this terminology comes from \cite{bk}.

Our first main result is the following theorem, proved in Section \ref{sec-proof-Theo1}. 

\begin{thm}\label{mainthm} 	Let $L$ be a finite-dimensional complex Lie algebra, $U(L)$ its universal enveloping algebras and $\R$ the class of all finite-dimensional representations of $U(L)$.
For any elements $p_1,\ldots,p_k \in U(L)$ the following are equivalent:
    \begin{enumerate}
    	\item\label{mainthm-pt1} $p_1,\ldots,p_k$ are linearly dependent.
    	\item\label{mainthm-pt2} $p_1,\ldots,p_k$ are $\R$-locally linearly dependent.
    	\item\label{mainthm-pt3} $p_1,\ldots,p_k$ are $\R$-locally directionally linearly dependent.
    \end{enumerate}
\end{thm}

The analogue of Theorem \ref{mainthm} for free algebras was proved in \cite{bk}. The analogue for the algebra $M_n(\CC)$ of all complex $n \times n$ matrices is trivial.
Namely, let $\pi$ be the direct sum of $n$ copies of the identity representation of $M_n(\CC)$ and let $v$ be the direct sum of all elements of the standard
basis of $\CC^n$. Then \eqref{mainthm-pt3} implies that $\pi(p_1)v,\ldots,\pi(p_n)v$ are linearly dependent which implies \eqref{mainthm-pt1} since each $\pi(p_i)v$ is just a vectorization of $p_i$. See also Lemma \ref{lem1t1} below.

Our second main result, whose proof is given in Section \ref{proof-of-sec2}, is:

\begin{thm}
\label{mainthm2} 
	Let $\sl_2$ be the Lie algebra of trace-zero $2 \times 2$ complex matrices and $\I$ the class of all finite-dimensional 
	irreducible representations of its universal enveloping algebra $U(\sl_2)$.
For any elements $p_1,\ldots,p_k \in U(\sl_2)$ the following are equivalent:
    \begin{enumerate}
        \item\label{mainthm2-pt1} There exist $z_1,\ldots,z_k$ in the center of $U(\sl_2)$ which are not all zero such that $z_1 p_1+\ldots+z_k p_k=0$. 
        \item\label{mainthm2-pt2} $p_1,\ldots,p_k$ are $\I$-locally linearly dependent.
        \item\label{mainthm2-pt3} $p_1,\ldots,p_k$ are $\I$-locally directionally linearly dependent.
    \end{enumerate}
\end{thm}

To obtain the analogue of Theorem \ref{mainthm2} for the enveloping algebra $U(\sl_3)$, which is our third main result, proved in Section \ref{proof-of-Th3},
we consider a smaller class of irreducible representations. Namely, for each $d \in \NN$ we define $\I_d$ to be the class of all finite-dimensional
irreducible representations of $\sl_3$ with highest weights $(m_1,m_2)$ satisfying $m_1\ge d$, $m_2\ge d$. 

\begin{thm}
\label{mainthm3} 
Let $\sl_3$ be the Lie algebra of trace-zero $3 \times 3$ complex matrices. 
For any elements $p_1,\ldots,p_k \in U(\sl_3)$ the following are equivalent:
    \begin{enumerate}
        \item\label{mainthm3-pt1} There exist $z_1,\ldots,z_k$ in the center of $U(\sl_3)$ which are not all zero such that $z_1 p_1+\ldots+z_k p_k=0$. 
        \item\label{mainthm3-pt2} There exists $d \in \NN$ such that $p_1,\ldots,p_k$ are $\I_d$-locally linearly dependent.
        \item\label{mainthm3-pt3} There exists $d \in \NN$ such that $p_1,\ldots,p_k$ are $\I_d$-locally directionally linearly dependent.
    \end{enumerate}
\end{thm}

Here is a list of a few results related to Theorem \ref{mainthm2} and \ref{mainthm3} that are either known or trivial:
    \begin{enumerate} 
        \item\label{rem-pt1} If $A=M_n(\CC)$ and $\C=\{\id\}$ then $\C$-LLD is equivalent to linear dependence but $\C$-LDLD is not
           as it is equivalent to the usual notion of locally linearly dependent matrices; see \cite{bs}.
           For $n \ge 2$ the coordinate matrices $E_{ij} \in M_n(\CC)$ are $\C$-LDLD although they are linearly independent.
           
        \item If $A=M_n(\CC[X_1,\ldots,X_m])$ and $\C=\{\ev_a \mid a \in \CC^m\}$ is the set of all evaluations at $n$-tuplets of complex numbers, then
		 $\C$-LLD is equivalent to linear dependence over $\CC[X_1,\ldots,X_n]$ (see below), but $\C$-LDLD is not (it suffices to consider constant matrices: see
		 \eqref{rem-pt1} above). 
       	  
       	  Pick any matrices $P_1,\ldots,P_k \in A$ and consider the matrix $P=[\mathbf{p}_1,\ldots,\mathbf{p}_k]$
       	   where $\mathbf{p}_i$ is the vectorization of $P_i$. 
       	   Note that $P_1,\ldots,P_k$ are $\C$-LLD iff for every $a \in \CC^n$ every maximal subdeterminant of $P(a)$ is zero iff every maximal subdeterminant of $P$ is zero
       	   iff $P_1,\ldots,P_k$ are linearly dependent over $\CC(X_1,\ldots,X_n)$.
       	          	   
       	 \item If $A=A_n(\CC)$ is the $n$-th Weyl algebra and $\mathcal C=\{\pi_0\},$ where $\pi_0$ is the Schr\" odinger representation of $A$,
			then $\C$-LLD and $\C$-LDLD are equivalent to linear dependence; see \cite{c}.
			 Recall that $A_n(\CC)$ has generators $x_1,\ldots,x_n,y_1,\ldots,y_n$ and relations
			$y_i x_j-x_j y_i=\delta_{ij}$, $x_ix_j=x_jx_i$ and $y_iy_j=y_jy_i$ for $i,j=1,\ldots,n$ and that $\pi_0$ acts on the vector space $\mathcal S(\RR^n)$ of rapidly decreasing 
			$C^{\infty}$-functions $f \colon \RR^n\to \CC$ by
			$(x_j f)(t)=t_j f(t)$ and $(y_i f)(t)=\frac{\partial f}{\partial t_i} (t)$, where $f\in \mathcal S(\RR^n)$ and 
			$t:=(t_1,\ldots,t_n)\in\RR^n$. Note that the center of $A$ is $\CC$; see \cite[Example 2.5.2]{sch}.
			   
\end{enumerate}

Recall from linear algebra that the span of elements $p_1,\ldots,p_k$ in a complex vector space is the set
$\span\{p_1,\ldots,p_k\}$ of all complex linear combinations of $p_1,\ldots,p_k$.
For an algebra $A$ and a class $\C$ of representations of $A$, two more notions of span of elements
$p_1,\ldots,p_k \in A$ will be used throughout the paper:\\

\noindent the $\C$-\textit{local linear span} of $p_1,\ldots,p_k$, denoted by $\loc_\C\{p_1,\ldots,p_k\}$, is the set of all $q\in A$ such that
\begin{equation}\label{loceq}\tag{A}
	\pi(q) \in \span\{\pi(p_1),\ldots,\pi(p_k)\} \quad\text{for all }\pi \in \C;\\\vspace{0.1cm} 
\end{equation}
and the $\C$-\textit{reflexive closure} of $p_1,\ldots,p_k$, denoted by $\rf_\C\{p_1,\ldots,p_k\}$, is the set of all $q\in A$ with
\begin{equation}\label{rfeq}\tag{B}
   \pi(q)v \in \span\{\pi(p_1)v,\ldots,\pi(p_k)v\}, \text{ for every } \pi:A\to \endo(V_\pi) \text{ in } \C \text{ and } v \in V_\pi.\\\vspace{0.1cm} 
\end{equation}

\noindent Clearly, $\span\{p_1,\ldots,p_k\} \subseteq \loc_\C\{p_1,\ldots,p_k\} \subseteq \rf_\C\{p_1,\ldots,p_k\}$, and Theorem \ref{mainthm} implies 
that $\span\{p_1,\ldots,p_k\} = \loc_\R \{p_1,\ldots,p_k\} = \rf_\R\{p_1,\ldots,p_k\}$ for every $p_1,\ldots,p_k\in U(L)$.

\noindent On the other hand, we do not have a similar result in $U(\sl_2)$ for $\rf_\I$ and $\loc_\I$. We will provide several counterexamples in Section \ref{sec-ref} (see Theorem \ref{I-ref-cl-sl2}).
For finite-dimensional complex solvable Lie algebras we will give explicit descriptions of $\rf_\I$ and $\loc_\I$ in Section \ref{sec-solv}.

Our motivation for studying $\loc$ and $\rf$ comes from their relation to nullstellensatz.
Namely, assume that the class $\C$ contains only finite-dimensional representations.
Then the properties \eqref{loceq} and \eqref{rfeq} are respectively equivalent to the properties (A') and (B') below: \\

\noindent (A')\; For every $\pi:A\to\endo(V_\pi)$ in $\C$ and a matrix $B \in \endo(V_\pi)$,  
\begin{equation*}\label{locnsatz}
	\tr(\pi(p_1)B)=\ldots=\tr(\pi(p_k)B)=0 \text{ implies } \tr (\pi(q)B)=0.
\end{equation*}


\noindent (B')\; For every $\pi \in \C$, $v \in V_\pi$ and $w \in V_\pi^\ast$,
\begin{equation*}\label{refnsatz}
	\langle \pi(p_1)v,w \rangle=\ldots=\langle \pi(p_k)v,w \rangle=0 \text{ implies } \langle \pi(q)v,w \rangle=0.
\end{equation*}
Here $V_\pi^\ast$ stands for the dual of $V_\pi$ and $\langle u,w \rangle=w(u)$. 
These equivalences are pretty easy to prove: the proof of the equivalence of (A) and (A')
uses the fact that the span of $\pi(p_1),\ldots,\pi(p_k)$ is equal to its second orthogonal complement in $\endo(V_\pi)$ with inner product defined by the trace map;
the proof of the equivalence of (B) and (B') is based on the span of $\pi(p_1)v,\ldots,\pi(p_k)v$ being equal to its second annihilator in $V_\pi$.\\

\noindent \textbf{Acknowledgement.} The authors would like to thank anonymous referee for a very detailed report and numerous suggestions which improved the exposition of the
manuscript.

\section{Proof of Theorem \ref{mainthm}}
\label{sec-proof-Theo1}

Let $\sl_n$ denote the Lie algebra of all complex $n \times n$ matrices with zero trace. 
A theorem of Ado, see \cite[2.5.6]{dix}, implies that for every finite-dimensional complex Lie algebra $L$ there exists an embedding $\iota \colon L\to \sl_n$ for some $n$. 

Let $U(L)$ be the universal enveloping algebra of $L$.
By the PBW theorem \cite[\S 17.3]{hum} $\iota$ induces an embedding of $U(L)$ into $U(\sl_n)$.
If $f_1,\ldots,f_{n^2-1}$ is a basis of $\sl_n$, then the monomials 
$f_1^{m_1}\cdots f_{n^2-1}^{m_{n^2-1}}, m_j\in \NN_0$, form a basis of $U(\sl_n)$.

We write $\R$ for the class of all finite-dimensional representations of $L$.
Proposition \ref{solv2simp} below reduces Theorem \ref{mainthm}
to a special linearly independent set in $\sl_n$.

\begin{proposition} 
\label{solv2simp}
The following statements are equivalent:
\begin{enumerate} 
\item\label{solv2simp-p1} 
	For every finite-dimensional Lie algebra $L$ over $\CC$ we have that every finite $\R$-locally directionally linearly dependent subset of $U(L)$ is linearly dependent.
\item\label{solv2simp-p2} 
	For every finite-dimensional Lie algebra $L$ over $\CC$ and every linearly independent set $p_1,\ldots,p_k \in U(L)$ 
	there exists $\pi:U(L)\to \endo(V_\pi)$ in $\R$ and a vector $v \in V_\pi$ such that $\pi(p_1)v,\ldots,\pi(p_k)v$ are linearly independent.
\item\label{solv2simp-p3} 
	For every $n,d \in \NN$ there exists a finite-dimensional representation $\pi_{n,d} \colon U(\sl_n)\to \endo(V_{\pi_{n,d}})$ and a vector $v_{n,d} \in V_{\pi_{n,d}}$ such that all vectors
	of the form 	
		$$\pi_{n,d}(f_1)^{m_{1}} \cdots \pi_{n,d}(f_{n^2-1})^{m_{n^2-1}}  v_{n,d}$$ 
	where $f_1,\ldots,f_{n^2-1}$ is a basis for $\sl_n$  and $\sum_{i=1}^{n^2-1} m_{i} \leq d$,
	are linearly independent.
\end{enumerate}
\end{proposition}

\begin{proof}
Clearly, \eqref{solv2simp-p1} is equivalent to \eqref{solv2simp-p2} and \eqref{solv2simp-p3} is a special case of \eqref{solv2simp-p2}. 

It remains to prove the implication \eqref{solv2simp-p3} $\Rightarrow$ \eqref{solv2simp-p2}. Let $L$ be a finite-dimensional complex Lie algebra
and let $p_1,\ldots,p_k \in U(L)$ be linearly independent. We first identify $U(L)$ with a subalgebra of $U(\sl_n)$ for some $n$. Then using the basis
$f_1,\ldots,f_{n^2-1}$ of $U(\sl_n)$, for each $\ell \in \NN$, let $W_\ell$ denote the subspace of $U(\sl_n)$ 
spanned by all elements $\prod_{i=1}^{n^2-1} f_i^{m_i}$ where $\sum_{i=1}^{n^2-1} m_i\leq \ell$.

Let $d\in \NN$ be such that $p_1,\ldots,p_k\in W_d$. Choose $p_{k+1},\ldots,p_N \in W_d$, so that $p_1,\ldots,p_N$  is a basis of $W_d$. 
Let $q_1,\ldots,q_N$ be another basis of $W_d$ consisting of all monomials $\prod_{i=1}^{n^2-1} f_i^{m_i}$ with $\sum_{i=1}^{n^2-1} m_i\leq d$.
By assumption \eqref{solv2simp-p3} there exists a representation  $\pi_{n,d} \colon U(\sl_n)\to \endo(V_{\pi_{n,d}})$ and a vector $v_{n,d} \in V_{\pi_{n,d}}$ 
such that the vectors $\pi_{n,d}(q_1)v_{n,d},\ldots,\pi_{n,d}(q_N)v_{n,d}$ are linearly independent. 
Claim \eqref{solv2simp-p2} will follow from the linear independence of $$\pi_{n,d}(p_1)v_{n,d},\ldots,\pi_{n,d}(p_N)v_{n,d}$$ which we now show.
There are $\gamma_{ij} \in \CC$ such that $p_i=\sum_{j=1}^N \gamma_{ij} q_j$ for $i=1,\ldots,N$. 
Assume that $\sum_{i=1}^N \alpha_i \pi_{n,d}(p_i)v_{n,d}=0$ for some $\alpha_i \in \CC$. Then
$\sum_{j=1}^N \beta_j \pi_{n,d}(q_j)v_{n,d}=0$ where $\beta_j=\sum_{i=1}^N \alpha_i \gamma_{ij}$
for every $j=1,\ldots,N$. Since $\pi_{n,d}(q_j)v_{n,d}$ are linearly independent
it follows that $\beta_j=0$ for $j=1,\ldots,N$. Since the matrix $[\gamma_{ij}]_{i,j}$ is invertible, it follows that
$\alpha_i=0$ for $i=1,\ldots,N$.
\end{proof}


Let $\rho_n \colon \sl_n \to\endo(\CC^n)$ be the standard representation of $\sl_n$ defined  by $\rho_n(X)u:=Xu$ for every $X \in \sl_n$ and $u\in \CC^n$.
Its unique extension to $U(\sl_n)$ will be denoted by the same symbol. Let $\pi_n=\bigoplus_{i=1}^n \rho_n$ be the direct sum of $n$ copies of $\rho_n$
and let $v=\bigoplus_{i=1}^n e_i$, where $e_1,\ldots,e_n$ is the standard basis of $\CC^n$. Note that $v$ belongs to $V:=\oplus_{i=1}^n \CC^n=\CC^{n^2}$ 
and that $\pi_n$ maps into $\endo(V)$. Let $f_1,\ldots,f_{n^2-1}$ be a basis of $\sl_n$.
The following is clear:

\begin{lemma}\label{lem1t1}
With the above notation, the vectors $v$, $\pi_n(f_1)v,\ldots,\pi_n(f_{n^2-1})v$ are linearly independent.
\end{lemma}

For every  $k \in \NN$ let $V^{\otimes k}$ be the $k$-th tensor power of $V$ and let $\sym^k(V)$
be the $k$-th symmetric power of $V$. Recall that $\sym^k(V)$ is the subset of $V^{\otimes k}$ consisting of all elements
that are invariant under the natural action of the symmetric group $S_k$ on $V^{\otimes k}$.
We define a representation $\sym^k(\pi_n) \colon \sl_n \to \endo(\sym^k (V))$ by
$$\sym^k(\pi_n)(x):= \sum_{i=0}^{k-1} I^{\otimes i} \otimes \pi_n(x) \otimes I^{\otimes (k-i-1)}$$
where $I\in \endo(V)$ is the identity. Its extension to $U(\sl_n)$ is unique and it will be denoted by the same symbol.

\begin{lemma}\label{same-degree-lin-ind}
Let $\pi_n \colon \sl_n \to \endo(V)$, $v \in V$ and $f_1,\ldots,f_{n^2-1} \in \sl_n$ be as in Lemma \ref{lem1t1}.
	Let $F_k$ be the subspace of $\Sym^k (V)$ generated by all elements of the form
		\begin{equation*}
	\sum_{\sigma\in S_k} u_{\sigma(1)}\otimes u_{\sigma(2)}\otimes\cdots\otimes u_{\sigma(k)}, 	\text{ where } u_1,\ldots,u_k \in V \text{ and } u_1=v. 
		\end{equation*}
	Then the vectors
		$$\sym^k(\pi_n)(f_{i_1}\ldots f_{i_k})  v^{\otimes k}, \text{ where } 1\leq i_1\leq i_2\leq \ldots\leq i_k\leq n^2-1,$$ are linearly independent in $\sym^k (V)/F_k$.
\end{lemma}

\begin{proof} 
	By Lemma  \ref{lem1t1}, the vectors $v_i:=\pi_{n}(f_i)v$, $1\leq i \leq n^2-1$, and $v_0=v$ are linearly independent.
	We have that
	\begin{equation}\label{tensor-calc}
		\sym^k(\pi_n)(f_{i_1}\cdots f_{i_k})v^{\otimes k} - 
		\sum_{\sigma\in S_k}v_{i_{\sigma(1)}}\otimes v_{i_{\sigma(2)}}\otimes\cdots\otimes v_{i_{\sigma(k)}} \in  F_k.
	\end{equation}	
	Note that the projection of the set 
	\begin{equation}\label{lin-ind-set}
		\Big\{\sum_{\sigma\in S_k} v_{i_{\sigma(1)}}\otimes v_{i_{\sigma(2)}} \otimes\cdots\otimes v_{i_{\sigma(k)}}  \colon
			 1\leq i_1\leq i_2\leq\ldots\leq i_k\leq n^2-1 \Big\}
	\end{equation}
	into the vector space $\Sym^k (V)/F_k$ is linearly independent.
	By \eqref{tensor-calc} and \eqref{lin-ind-set} the conclusion of the lemma follows.
\end{proof}

\begin{proof}[Proof of Theorem \ref{mainthm}] 
	The implications $\eqref{mainthm-pt1} \Rightarrow \eqref{mainthm-pt2} \Rightarrow \eqref{mainthm-pt3}$ are trivial.

	It remains to prove the implication $\eqref{mainthm-pt3} \Rightarrow \eqref{mainthm-pt1}$.
	It suffices to prove	statement \eqref{solv2simp-p3} of Proposition \ref{solv2simp}. 
	Fix $n,d\in \NN$. With the notation from Lemma \ref{same-degree-lin-ind} we define a representation 
	$\pi_{n,d}:=\bigoplus_{k=1}^d \sym^k(\pi_n)$ and a vector $v_{n,d}:=\bigoplus_{k=1}^d v^{\otimes k}$. 
	
  To prove that the vectors
		\begin{equation}\label{lin-ind-pred}
			\pi_{n,d}(f_1)^{m_{1}} \cdots \pi_{n,d}(f_{n^2-1})^{m_{n^2-1}}  v_{n,d}
			\quad\text{where}\quad 
			m_1+\ldots+m_{n^2-1}\le d
		\end{equation}
	are linearly independent we assume that
		\begin{equation}\label{lin-ind-eq}
			\sum_{\sum_{i=1}^{n^2-1}m_i\leq d} \lambda_{m_1,\ldots,m_{n^2-1}}\pi_{n,d}(f_1)^{m_{1}} \cdots \pi_{n,d}(f_{n^2-1})^{m_{n^2-1}}  v_{n,d}=0.
		\end{equation}
  Project this onto
	$$\bigoplus_{k=1}^d \sym^k(V) /  \Big(\bigoplus_{k=1}^{d-1} \sym^k(V) \oplus F_d\Big) \cong \sym^d(V)/F_d$$ 
  to conclude, by Lemma \ref{same-degree-lin-ind},	that  $\lambda_{m_1,\ldots,m_{n^2-1}}=0$ whenever $\sum_{i=1}^{n^2-1}m_i=d$. 
  Repeating this argument for $d-1, d-2, \ldots$ in place of $d$, we prove that 
  $\lambda_{m_1,\ldots,m_{n^2-1}}=0$ for all $\sum_{i=1}^{n^2-1} m_i\leq d$.
\end{proof}

\section{Proof of Theorem \ref{mainthm2}}
\label{proof-of-sec2}

\subsection{Irreducible representations of $\sl_2$}
\label{sec-irr-sl2}

The main result of this subsection, Proposition \ref{irr-repr-prop}, describes an irreducible representation of the Lie algebra $\sl_2$ of $2\times 2$ complex traceless matrices and a vector $v$ making 
monomials of the form \eqref{vec-lin-ind-2} linearly independent. This result will be needed in the proof of Theorem \ref{mainthm2}, given in Subsection \ref{sec-I-LDLD-sl2} below.

Let $e_1,\ldots,e_k$ be the standard basis of $\CC^k$, let $E_{ij}$, $1\leq i,j\leq 2$, be the standard basis of $M_2(\CC)$ and let
$X:=E_{12}$, $Y:=E_{21}$ and $H:=E_{11}-E_{22}$ be the standard basis of $\sl_2$.
Recall \cite[\S 1.8]{dix} that for every $k\in \NN$ there is a unique (up to equivalence) irreducible representation $\rho_k:\sl_2\to \endo(\CC^k)$
defined by 
	$$\rho_k(X)e_i=x_{k,i-1}e_{i-1},\quad \rho_k(Y)e_i=y_{k,i}e_{i+1}, \quad \rho_k(H)e_i=h_{k,i}e_{i},$$
where
	\begin{eqnarray*}	
		x_{k,i}	&:=&	\left\{\begin{array}{rr}
						k-i,		&	\text{if}\quad i=1,\ldots,k-1	\\
						0,		&	\text{otherwise}\\
					\end{array} \right.,\\
		y_{k,i}	&:=& \left\{\begin{array}{rr}
						i,		&	\text{if}\quad i=1,\ldots,k-1	\\
						0,		&	\text{otherwise}	
					\end{array} \right.,\\
		h_{k,i}	&:=& (k+1-2i),		\quad\text{for}\quad i=1,\ldots,k.
	\end{eqnarray*}

\noindent We denote by $0_\ell$ a sequence of $\ell$ zeroes.

\begin{proposition}
	\label{irr-repr-prop}
	Assume the notation as above.
	For every $t \in \NN\cup\{0\}$ and a vector $v\in \CC^{(d+1)^2+t}$, $d\in \NN$, of the form 
		$$v=[0_d,1,0_d,0_{d-1},1,0_{d-1},\ldots,0_2,1,0_2,0_1,1,0_1,1,0_t]^T$$
	all vectors of the form 
	\begin{equation}\label{vec-lin-ind-2}
		\rho_{(d+1)^2+t}(X)^{m_{1}}\rho_{(d+1)^2+t}(Y)^{m_{2}}\rho_{(d+1)^2+t}(H)^{m_{3}}v
	\end{equation} 
	with $m_1,m_2,m_3\in \NN_0$, $0\leq m_{1}+m_{2}+m_3\le d$ and $m_1 m_2=0$ are linearly independent.
\end{proposition} 

\begin{proof}
Let $e_1,\ldots,e_{(d+1)^2+t}$ be the standard basis of $\CC^{(d+1)^2+t}$. Then
\begin{equation}\label{veeq}
	v=e_{i_1}+\ldots+e_{i_{d+1}}
\end{equation}
 where $i_1=d+1$ and $i_{k+1}=i_k+2(d-k+1)$ for $k=1,\ldots,d$.
Note that $i_{d+1}=(d+1)^2$. 

For every $k=-d,\ldots,d$ and $\ell=0,\ldots,d$ we write
\begin{equation}\label{zdefeq}
	Z_k=\left\{ \begin{array}{cc} \rho_{(d+1)^2+t}(X)^k & \text{ if } k>0 \\ 1 & \text{ if } k=0 \\ \rho_{(d+1)^2+t}(Y)^{-k} & \text{ if } k<0 \end{array} \right. 
		\quad \text{ and } \quad H_\ell=\rho_{(d+1)^2+t}(H)^\ell.
\end{equation}
To prove that all $Z_k H_\ell v$ with $|k|+\ell \le d$ are linearly independent we assume that
	\begin{equation}\label{lin-comb-zero}
		\sum_{\substack{|k|+\ell \le d}} \alpha_{k,\ell} Z_k H_\ell v=0.
	\end{equation}
Since $(d+1)^2+t$ is fixed in the proof, we abbreviate $x_j:=x_{(d+1)^2+t,j}$, $y_j:=y_{(d+1)^2+t,j}$
and $h_j:=h_{(d+1)^2+t,j}$.
Thus
\begin{equation}\label{zheq}
	Z_k H_{\ell} e_j = z_{j,k} (h_j)^{\ell} e_{j-k}
\end{equation}
where $z_{j,k}=0$ if $j-k \not\in \{1,\ldots,(d+1)^2+t\}$ while in other cases 
\begin{equation*}
	z_{j,k}=\left\{ \begin{array}{cc} x_{j-k} \cdots x_{j-1} & \text{ if } k>0 \\ 1 & \text{ if } k=0 \\ y_{j} \cdots y_{j-k-1} & \text{ if } k<0 \end{array} \right.
\end{equation*}
Since $x_j$ and $y_j$ are nonzero for $j=1,\ldots,(d+1)^2+t-1$, it follows that $z_{j,k}$ are also nonzero when $1\leq j-k \leq (d+1)^2+t$.
If we substitute \eqref{veeq} and \eqref{zheq} into \eqref{lin-comb-zero}, we get
\begin{equation}\label{longeq}
	\sum_{k=-d}^d  \sum_{r=1}^{d+1} (\sum_{\ell=0}^{d-|k|} \alpha_{k,\ell} (h_{i_r})^\ell) z_{i_r,k}  e_{i_r-k}=0.
\end{equation}

We prove by backward induction on $|k|$ that the equation \eqref{longeq} implies $\alpha_{k,\ell}=0$ for all $k$ and $\ell$ such that $|k|+\ell \le d$.
This means we prove:
\begin{itemize}
	\item \textbf{Induction base:} $\alpha_{d,0}=\alpha_{-d,0}=0$.
	\item \textbf{Induction step:} Fix $m \in  \{0,\ldots,d-1\}$. Suppose $\alpha_{k,\ell}=0$ for $|k| \ge m+1$ and prove that $\alpha_{k,\ell}=0$ for $|k|=m$.
\end{itemize}

To establish the base of induction we first compute the coefficient of $e_1$ in \eqref{longeq}. Note that $e_{i_r-k}=e_1$ iff $r=1$ and $k=d$,
so that $|k|+\ell \le d$ forces $\ell=0$. 
Since $z_{i_1,d} \ne 0$ and $h_{i_1} \ne 0$, it follows that $\alpha_{d,0}=0$. Next we compute the coefficient of $e_{2d+1}$ in \eqref{longeq}. Note that $e_{i_r-k}=e_{2d+1}$ iff $r=1, k=-d$ or $r=2,k=d$. 
In both cases, it follows that $\ell=0$. Since $\alpha_{d,0}=0$ and $z_{i_1,-d} \ne 0$ and $h_{i_1} \ne 0$, it follows that $\alpha_{-d,0}=0$.

To prove induction step we assume that $\alpha_{k,\ell}=0$ for every $k$ with $|k| \ge m+1$. Then equation \eqref{longeq} implies that
\begin{equation}\label{shorteq}
	\sum_{k=-m}^m  \sum_{r=1}^{d+1} (\sum_{\ell=0}^{d-|k|} \alpha_{k,\ell} (h_{i_r})^\ell) z_{i_r,k}  e_{i_r-k}=0.
\end{equation}
Suppose that $s \in \{1,\ldots,d-m+1\}$. We claim that equation \eqref{shorteq} contains only one term with $e_{i_s-m}$ and only one term with $e_{i_s+m}$.
Namely, if $i_r-k=i_s-m$ for some $r=1,\ldots,d+1$ and $k=-m,\ldots,m$ then $m-k =i_s-i_r$. Clearly, $0 \le m-k \le 2m$, so $s \ge r$.
If $r=s$ then $k=m$ and we are done. Otherwise, $2m \ge i_s-i_r \ge i_s-i_{s-1} = 2(d-s+2)$ which implies that $s \ge d-m+2$. The other case is similar.
It follows that for every $s=1,\ldots,d-m+1$ 
\begin{equation}
	(\sum_{\ell=0}^{d-m} \alpha_{m,\ell} (h_{i_s})^\ell) z_{i_s,m} =0 \quad \text{ and } \quad (\sum_{\ell=0}^{d-m} \alpha_{-m,\ell} (h_{i_s})^\ell) z_{i_s,-m} =0 
\end{equation}
We divide out by $z_{i_s,m}$ and $z_{i_s,-m}$ to obtain two Vandermonde systems
\begin{equation}
	\sum_{\ell=0}^{d-m} \alpha_{m,\ell} (h_{i_s})^\ell =0 \quad \text{ and } \quad \sum_{\ell=0}^{d-m} \alpha_{-m,\ell} (h_{i_s})^\ell =0, \quad
	\text{for } s=1,\ldots,d-m+1.
\end{equation}
Since the $h_{i_s}$ are distinct for different $s$, the Vandermonde coefficient matrices in both are invertible. It follows that
\begin{equation}
\alpha_{m,0}=\ldots=\alpha_{m,d-m}=0 \quad \text{ and } \quad \alpha_{-m,0}=\ldots=\alpha_{-m,d-m}=0
\end{equation}
which completes the proof of the induction step.
\end{proof}

\subsection{$\I$-local directional linear dependence  in $\sl_2$}
\label{sec-I-LDLD-sl2}

In this subsection we prove Theorem \ref{mainthm2}, which is a characterization of the situation when finitely many elements of $U(\sl_2)$ are $\I$-locally directionally linearly independent,
where $\I$ stands for the class of all finite-dimensional irreducible representations of $U(\sl_2)$.

Recall from the previous subsection that $E_{ij}$, $1\leq i,j\leq 2$, is the standard basis of $M_2(\CC)$, and $X:=E_{12}$, $Y:=E_{21}$, $H:=E_{11}-E_{22}$
is the standard basis of $\sl_2$.
Let $\rho_k$, for $k\in \NN$, be the unique (up to equivalence) irreducible representation of $\sl_2$ of dimension $k$. 
The element
	$$C:=XY+\frac{1}{2}H^2+YX=2XY+\frac{1}{2}H^2-H$$ 
of the enveloping algebra $U(\sl_2)$ is called the \textit{Casimir element}. It is well-known that $C$ generates the center $Z$ of $U(\sl_2)$, i.e., $Z=\CC[C]$, and that
$\rho_k(C)=\frac{1}{2}(k^2-1) I_k$ where $I_k$ is the identity matrix of size $k$ (see \cite{hum}).
  We write $c_k:=\frac{1}{2}(k^2-1)$ for all $k\in \NN$.
Moreover, every element $p\in U(\sl_2)$ can be written in the form $p=\sum_{i=1}^m f_i s_i$ where $f_i$ are monomials of the form $X^{i_1}Y^{i_2}H^{i_3}$ with $i_1 i_2=0$ and $s_i\in \CC[C]$ are central elements.


Before we proceed with the proof of Theorem \ref{mainthm2}, we need the following lemma:

\begin{lemma} \label{technical-lemma}
	Suppose $u_1,\ldots,u_k\in \CC(z)^\ell$, for $k,\ell\in \NN$, are linearly dependent for infinitely many complex values of $z$. 
	Then they are linearly dependent over $\CC(z)$.
\end{lemma}

\begin{proof}
	Assume to the contrary that $u_1,\ldots,u_k$ are linearly independent over $\CC(z)$. Then we can add vectors $u_{k+1},\ldots,u_\ell\in \CC(z)^\ell$ such that $u_1,\ldots,u_\ell$ form a basis
	for  $\CC(z)^\ell$ over $\CC(z)$. The determinant of the matrix with columns $u_1,\ldots,u_\ell$ is a non-zero rational function $\frac{p(z)}{r(z)}\in \CC(z)$ which has only finitely many zeros, a contradiction 
	with the hypothesis that infinitely many evaluations of $u_1,\ldots,u_k$ are $\CC$-linearly dependent.
\end{proof}

\begin{proof}[Proof of Theorem \ref{mainthm2}]
	To prove the implication $\eqref{mainthm2-pt1} \Rightarrow \eqref{mainthm2-pt2}$, we first divide out the greatest common divisor of $z_1,\ldots,z_k\in \CC[C]$ from
 	the equation $\sum_{i=1}^k z_ip_i=0$. Hence, we can assume WLOG 
	that $z_1,\ldots,z_k$ do not have a common zero. Applying each $\rho_n \in \I$, for $n\in \NN$, to $\sum_{i=1}^k z_ip_i=0$	
	one gets
	$0=\sum_{i=1}^k z_i(c_n)\rho_n(p_i)$. Since $z_1,\ldots, z_k$ are without common zeroes, this linear combination is non-trivial and hence $p_1,\ldots,p_k$ are 
	$\I$-locally linearly dependent.

	The implication $\eqref{mainthm2-pt2} \Rightarrow \eqref{mainthm2-pt3}$ is trivial.

	It remains to prove the implication $\eqref{mainthm2-pt3} \Rightarrow \eqref{mainthm2-pt1}$.
	We write $p_j=\sum_{i=1}^m f_i t_{ij}$, $m\in \NN$, where $t_{ij}\in\CC[C]$ are central elements and 
	$f_i$ are different monomials of the form $X^{i_1}Y^{i_2}H^{i_3}$ with $i_1,i_2,i_3\in \NN_0$ and $i_1 i_2=0$.
	By Proposition \ref{irr-repr-prop} for all $n\in \NN_0$ sufficiently large there exist vectors $v_n\in V_{\rho_n}$ 
	such that vectors $\rho_n(f_i)v_n$, $i=1,\ldots,m$, are linearly independent.
	Therefore, for those $n$, the vectors $\rho_n(p_1)v_n,\ldots, \rho_n(p_k)v_n,$ are linearly dependent if and only if
	the vectors $[t_{1j}(c_n),\ldots,t_{mj}(c_n)]^T, j=1,\ldots,k$, are linearly dependent.
	Since this is true for infinitely many $n$-s, this implies by Lemma \ref{technical-lemma} that the vectors $[t_{1j}(C),\ldots,t_{mj}(C)]^T, j=1,\ldots,k$, are 
	$\CC(C)$-linearly dependent
	and hence there exist $v_j(C)\in \CC(C)$, $j=1,\ldots,k$, not all zero such that 
	$0=\sum_{j=1}^k v_j(C) [t_{1j}(C),\ldots,t_{mj}(C)]^T$.
	Multiplying by the least common denominator $z_0\in \CC[C]$ of nonzero $v_1,\ldots,v_k$ we obtain
		$0=\sum_{j=1}^k z_j [t_{1j}(C),\ldots,t_{mj}(C)]^T$
	for some $z_1,\ldots,z_k\in \CC[C]$, not all zero and hence $0=z_1 p_1+\ldots+z_k p_k.$ 
\end{proof}

\section{Reflexive closures}

\subsection{Reflexive closures in $\sl_2$}
\label{sec-ref}

Assume the notation from the previous section. Let $q,p_1,\ldots,p_k$ be elements of $U(\sl_2)$. 
Theorem \ref{I-ref-cl-sl2} gives a closely related sufficient condition \eqref{cl1-I-rc-sl2} and a necessary condition \eqref{cl4-I-rc-sl2} for $q$ to belong to the $\I$-local span, resp.\ the $\I$-reflexive closure, of $p_1,\ldots,p_k$.
The conditions differ only in the assumptions on the zero set of the central element $z_0$.
 
\begin{thm}\label{I-ref-cl-sl2}
	Let $q, p_1,\ldots,p_k$ be elements of $U(\sl_2)$ and consider the following statements:
	\begin{enumerate} 
		\item\label{cl1-I-rc-sl2} 
			There exist central elements $z_0, z_1,\ldots,z_k\in \CC[C]$ such that $z_0$ is nonzero,
			$z_0(c_n)\neq 0$ if $\rho_n(q)\neq 0$ and 
				$z_0 q=z_1 p_1+\ldots+z_k p_k.$
		\item\label{cl2-I-rc-sl2} 
			$q\in  \loc_\I\{p_1,\ldots,p_k\}.$ 
		\item\label{cl3-I-rc-sl2} 
			$q\in  \rf_\I\{p_1,\ldots,p_k\}.$ 
		\item\label{cl4-I-rc-sl2} 
			There exist central elements $z_0, z_1,\ldots,z_k\in \CC[C]$ such that $z_0$ is nonzero and
			$z_0 q=z_1 p_1+\ldots+z_k p_k.$
	\end{enumerate}
	Then $\eqref{cl1-I-rc-sl2}\Rightarrow \eqref{cl2-I-rc-sl2}\Rightarrow \eqref{cl3-I-rc-sl2}\Rightarrow \eqref{cl4-I-rc-sl2}$
	and the reverse implications do not hold.
\end{thm}

The proof of Theorem \ref{I-ref-cl-sl2} uses the following trivial consequence of Lemma \ref{technical-lemma}.

\begin{lemma} \label{technical-lemma-2}
	Suppose $s,u_1,\ldots,u_k\in \CC(z)^\ell$ for $k,\ell\in \NN$, have the property that
	$s(t)\in \span_{\CC}\{u_1(t),\ldots,u_k(t)\}$ for infinitely many $t\in\CC$.
	 Then $$s(z)\in \span_{\CC(z)}\{u_1(z),\ldots,u_k(z)\}.$$
\end{lemma}


\begin{proof}[Proof of Theorem \ref{I-ref-cl-sl2}]
	To prove $\mathrm{\eqref{cl1-I-rc-sl2}}\Rightarrow\mathrm{\eqref{cl2-I-rc-sl2}}$ note that $z_0 q=z_1 p_1+\ldots+z_k p_k$ implies that 
		$z_0(c_n)\rho_n(q)=z_1(c_n)\rho_n(p_1)+\ldots+z_k(c_n)\rho_n(p_k).$
	If $\rho_n(q)=0$, then clearly $\rho_n(q)\in \span\{\rho_{n}(p_1),\ldots,\rho_n(p_k)\}$. Otherwise $\rho_n(q)\neq 0$ which 
	implies by assumption that $z_0(c_n)\neq 0$ and hence again $\rho_n(q)\in \span\{\rho_{n}(p_1),\ldots,\rho_n(p_k)\}$.
	
	The implication $\mathrm{\eqref{cl2-I-rc-sl2}}\Rightarrow \mathrm{\eqref{cl3-I-rc-sl2}}$ is trivial. 
	
	The proof of $\mathrm{\eqref{cl3-I-rc-sl2}}\Rightarrow\mathrm{\eqref{cl4-I-rc-sl2}}$ is analogous to the proof of the implication  $\eqref{mainthm2-pt3}\Rightarrow \eqref{mainthm2-pt1}$ in Theorem 
	\ref{mainthm2}
	only that we use Lemma \ref{technical-lemma-2} instead of Lemma \ref{technical-lemma}.
%
	
	It remains to construct counterexamples for the reverse implications. To prove $\mathrm{\eqref{cl1-I-rc-sl2}}\not\Leftarrow\mathrm{\eqref{cl2-I-rc-sl2}}$ take
		$q=H$,  $p_1 = C X^2 + c_2 H$ and $p_2 = c_2 X^2 + C H$.
	First, we prove that $q\in  \loc_\I\{p_1,p_2\}.$ Since $\rho_2(X^2)=0$ we have
	$\rho_2(p_1)=\rho_2(p_2)=c_2\rho_2(H)=c_2\rho_2(q)$, so $\rho_2(q)\in \span\{\rho_2(p_1),\rho_2(p_2)\}$.
	For $n>2$ we have 
		$(c_2^2-c_n^2) \rho_n(q) = c_2 \rho_n(p_1)- c_n\rho_n(p_2),$
	which also implies that $\rho_n(q)\in \span\{\rho_n(p_1),\rho_n(p_2)\}.$
	Second, we show that each triplet of central elements $z_0$, $z_1$, $z_2$ that satisfy $z_0q=z_1p_1+z_2p_2$ must have that
	$z_0(c_2)= 0$. By comparing the coefficients at $X^2$ and $H$ we get the system
	$0=z_1C+z_2c_2$ and $z_0=z_1c_2+z_2C$.
	Hence $c_2z_0=z_1(c_2^2-C^2)$ and $z_0(c_2)=0$.
	
	To prove $\mathrm{\eqref{cl2-I-rc-sl2}}\not\Leftarrow\mathrm{\eqref{cl3-I-rc-sl2}}$ take 
		$q=X$, $p_1 = I+H$, $p_2 = X+Y$ and $p_3=(C-c_2)X.$
	Clearly 
		$$\rho_2(q)=E_{12} \notin \span\{2 E_{11},E_{12}+E_{21},0\}=\span\{\rho_2(p_1),\rho_2(p_2),\rho_2(p_3)\},$$
	which implies that $q \not\in \loc_\I\{p_1,p_2,p_3\}$. 
	Since
		$$[y,0]^T \in \span\{2 [x,0]^T,[y,x]^T\}$$
	for every $x$ and $y$, we have that 
		$\rho_2(q)v \in \span\{\rho_2(p_1)v,\rho_2(p_2)v,\rho_2(p_3)v\}.$
	Clearly, we also have that 
		$\rho_n(q)v=\frac{1}{c_n-c_2}\rho_n(p_3)v$
	for all  $n\geq 3$ and  $v\in \CC^n$, which implies that $q \in \rf_I\{p_1,p_2,p_3\}$.
	
	To prove $\mathrm{\eqref{cl3-I-rc-sl2}}\not\Leftarrow\mathrm{\eqref{cl4-I-rc-sl2}}$ take $q=I$ and $p=(C-c_2)I$
	and notice that $(C-c_2)q=p$ but $q \not\in \rf_\I\{p\}$ since 
		$\rho_2(q)e_1=e_1 \not\in \{0\}= \span\{\rho_2(p)e_1\}.$
\end{proof}

As seen in the proof \eqref{cl4-I-rc-sl2} of Theorem \ref{I-ref-cl-sl2} does not suffice to conclude $q\in  \rf_\I\{p_1,\ldots,p_k\}.$ 
The failure of the reverse implications in Theorem \ref{I-ref-cl-sl2} are caused by representations $\rho_n$ of small dimension $n$. The following theorem says that, for
$n$ big enough, the same reverse implications hold true.

\begin{thm}
	Let $q, p_1,\ldots,p_k$ be elements from $U(\sl_2)$. Then the following statements are equivalent:
	\begin{enumerate}
		\item\label{cl4-I-rc2-sl2} 
			There exist central elements $z_0, z_1,\ldots,z_k\in \CC[C]$ such that $z_0\neq 0$ and
			$z_0 q=z_1 p_1+\ldots+z_k p_k.$
		\item\label{cl5-I-rc2-sl2} 
			$\rho_n(q) \in \span\{\rho_n(p_1),\ldots,\rho_n(p_k)\}$ for every $n\in\NN$ big enough.
		\item\label{cl6-I-rc2-sl2} 
			$\rho_n(q)v \in \span\{\rho_n(p_1)v,\ldots,\rho_n(p_k)v\}$ for every $n\in\NN$ big enough and every vector $v$.
	\end{enumerate}
\end{thm}

\begin{proof}
	To prove $\eqref{cl4-I-rc2-sl2}\Rightarrow \eqref{cl5-I-rc2-sl2}$ one takes $n$ big enough such that $z_0(c_r)\neq 0$ for every $r\geq n$.
	Notice that for all such $r$ we have that $\rho_r(z_0)=z_0(c_r)\neq 0$ and hence 
		$\rho_r(q)=\frac{1}{\rho_r(z_0)}\sum_{i=1}^k \rho_r(z_i)\rho_r(p_i)$.
	The implication $\mathrm{\eqref{cl5-I-rc2-sl2}}\Rightarrow \mathrm{\eqref{cl6-I-rc2-sl2}}$ is clear.
	The implication $\mathrm{\eqref{cl6-I-rc2-sl2}}\Rightarrow \mathrm{\eqref{cl4-I-rc2-sl2}}$ follows easily from the proof of the implication 
	$\mathrm{\eqref{cl3-I-rc-sl2}}\Rightarrow \mathrm{\eqref{cl4-I-rc-sl2}}$ in Theorem \ref{I-ref-cl-sl2}, since for $n$
	big enough, there exist vectors $v_n\in V_{\rho_n}$ such that $\rho_n(q)v_n\in \span\{\rho_n(p_i)v_n\colon i=1,\ldots,k\}$.
\end{proof}

\subsection{Reflexive closures in solvable Lie algebras}
\label{sec-solv}

By Lie's theorem \cite[Theorem 9.11]{fh}, every irreducible representation $\pi$ of a finite-dimensional complex solvable Lie algebra $L$ is one-dimensional.
It follows that $\pi$ annihilates $L_1:=[L,L]$, hence it factors through the 
abelian Lie algebra $L/L_1$. Let $R$ be the left (equivalently the right) ideal of $U(L)$
generated by $L_1$. By \cite[Proposition 2.2.14]{dix}, the canonical homomorphism
from $U(L)$ to $U(L/L_1)$ is surjective with kernel $R$ and so $U(L)/R \cong U(L/L_1)$.
Clearly, every irreducible representation of $U(L)$ factors through $U(L)/
R$.

\begin{thm}
Let $L$ be a finite-dimensional complex solvable Lie algebra and $R$ the two-sided ideal of $U(L)$ generated by $L_1=[L,L]$.
Pick $p_1,\ldots,p_k,q \in U(L)$ and write $I$ for the two-sided ideal of $U(L)$ generated by $p_1,\ldots,p_k$.
The following are equivalent:
\begin{enumerate}
\item\label{solv-pt1} For some $n \in \NN$ we have that $q^n \in I+R$.
\item\label{solv-pt2} Every irreducible representation of $U(L)$ which anihilates $p_1,\ldots,p_k$ also annihilates $q$.
\item\label{solv-pt3} $q \in \loc_\I\{p_1,\ldots,p_k\}$.
\item\label{solv-pt4} $q \in \rf_\I\{p_1,\ldots,p_k\}$.
\end{enumerate}
\end{thm}

\begin{proof}
The equivalence of \eqref{solv-pt1} and \eqref{solv-pt2} follows from Hilbert's Nullsellensatz and $U(L)/R \cong U(L/L_1)$. Namely, since  $U(L/L_1)$  is isomorphic to a polynomial algebra,
the following are equivalent for any $p_1',\ldots,p_k',q' \in U(L/L_1)$:
\begin{itemize}
\item $q'$ belongs to the radical of the ideal generated by $p_1',\ldots,p_k'$.
\item Every character $\phi$ of $U(L/L_1)$ which anihilates $p_1',\ldots,p_k'$ also anihilates $q'$.
\end{itemize}

The equivalence of \eqref{solv-pt2} and \eqref{solv-pt3} follows from the trivial observation that
for complex numbers $\alpha_1,\ldots,\alpha_k,\beta$ we have that $\beta \in \span\{\alpha_1,\ldots,\alpha_k\}$
iff  $\alpha_1=\ldots=\alpha_k=0$ implies $\beta=0$. 

Since all irreducible representations are one-dimensional, \eqref{solv-pt3} is equivalent to \eqref{solv-pt4}.
\end{proof}

%

\section{Proof of Theorem \ref{mainthm3}}
\label{proof-of-Th3}

\subsection{$\I$-local directional linear dependence  in $\sl_3$}
\label{usl3-ind}

	The Lie algebra of all trace-zero complex $3 \times 3$ matrices is denoted by $\sl_3$. We refer the reader to \cite[Chapter 6]{hall} for the theory of representations
of $\sl_3$; here we write the basics.
The standard basis of $\sl_3$ is 
\begin{equation}\label{basis-of-sl-3}
\begin{split}
	X_1	&:=E_{12},\; X_2:=E_{23},\; X_3:=E_{13},\; Y_1:=E_{21},\; Y_2:=E_{32},\\
	Y_3	&:=E_{31},\; H_1:=E_{11}-E_{22},\; H_2:=E_{22}-E_{33}.
\end{split}
\end{equation}
We write $V_1=V_2=\CC^3$. Let $e_1,e_2,e_3$ be the standard basis of $V_1$ and let $f_1=e_3,f_2=-e_2,f_3=e_1$ be a basis of $V_2$.
The action of $\sl_3$ on $V_1$ is defined by $\pi_1(Z)v:=Zv$ and its action on $V_2$ is defined by $\pi_2(Z)v:=-Z^Tv$. 
(Note that $\pi_1$ is the standard representation and $\pi_2$ is its adjoint.)
For every $m_1,m_2 \in \NN$, we identify the $m_1$-th symmetric power $\Sym^{m_1}(V_1)$ of $V_1$ with the vector space of all homogeneous polynomials of degree $m_1$ in $e_1,e_2,e_3$. 
Similarly, we identify $\Sym^{m_2}(V_2)$ with the vector space of all homogeneous polynomials of degree $m_2$ in $f_1,f_2,f_3$.
Let $\psi_1$ be the representation of $\sl_3$ on $\Sym^{m_1}(V_1)$ defined by
	$$\psi_1(e_{i_1}e_{i_2} \cdots e_{i_m}):=\sum_{j=1}^{m} e_{i_1}\cdots e_{i_{j-1}} \pi_1(e_{i_j}) e_{i_{j+1}}\cdots e_{i_m}.$$
$\psi_2$ is defined analogously. The representations $\psi_1$ and $\psi_2$ are irreducible 
but their tensor product $\psi:=\psi_1 \otimes \psi_2$, defined by
	$$\psi(v_1 \otimes v_2):=\psi_1(v_1) \otimes v_2+v_1 \otimes \psi_2(v_2)$$
is not irreducible. Let $W$ be the subspace of $\Sym^{m_1}(V_1) \otimes \Sym^{m_2}(V_2)$ generated by all elements of the form
	\begin{equation} \label{irr-repr}
		v_{i,j,k}:=\psi(Y_1^i Y_2^j Y_3^k)(e_1^{m_1} \otimes f_1^{m_2}), \quad i,j,k \in \NN_0.
	\end{equation}
It turns out that $W$ is an invariant subspace for $\psi(\sl_3)$ and the subrepresentation $\pi_{m_1,m_2} := \psi|_W$
is irreducible. Recall that a \emph{weight} of a representation $\pi$ is a pair of integers $z_1$, $z_2$ such that $\pi(H_i)v=z_i v $ for $i=1,2$ where 
$v$ is some nonzero vector, called a \emph{weight vector}. The weight $(m_1,m_2)$ is the \emph{highest weight} if for every weight $(m_1',m_2')$ we have
	$$(m_1,m_2)-(m_1',m_2')=a(2,-1)+b(-1,2)$$ 
for some $a,b\geq 0$. The highest weight of the representation with the irreducible subspace generated by \eqref{irr-repr} is $(m_1,m_2)$
and its highest weight vector is $v:=e_1^{m_1} \otimes f_1^{m_2}.$ 

In order to prove an analogue of Proposition \ref{irr-repr-prop}, we start with the following proposition.

\begin{proposition} \label{nonzero-proposition}
	For every $d,m_1,m_2 \in \NN_0$ with $m_1 \ge d$ and $m_2 \ge d$, the vectors
		$v_{k,\ell,m}$ with $k,\ell,m\in \NN_0$ such that $k+\ell+m\leq d$ are linearly independent.
\end{proposition}

\begin{proof}
Denote $S_d:=\{(k,\ell,m)\in \NN_0^3\colon k+\ell+m\leq d\}$ and assume that 
	\begin{equation}\label{eq-lin-ind}
		\sum_{(k,\ell,m)\in S_d} \alpha_{k,\ell,m} v_{k,\ell,m}=0
	\end{equation}
for some $\alpha_{k,\ell,m}\in \RR$.
We have to prove that each $\alpha_{k,\ell,m}$ is zero. 
After a short computation which depends on the formula
	\begin{equation*}
		\pi_{m_1,m_2}(Y_i^j)(u_1 \otimes u_2) =\sum_{q=0}^j \binom{j}{q} \psi_1(Y_i^q)u_1 \otimes \psi_2(Y_i^{j-q})u_2
	\end{equation*}
for each $i$ and $j$ we get that
	\begin{equation}\label{allveceq}
		v_{k,\ell,m}=
			\sum_{t=0}^m \sum_{s=0}^k \beta_{s,t}^{k,\ell,m} e_1^{m_1-s-t} e_2^s e_3^t \otimes f_1^{m_2+t-\ell-k} f_2^{\ell-k+s} f_3^{m+k-s-t} 	
	\end{equation}
where $\beta_{s,t}^{k,\ell,m}\in \RR$ and in particular $\beta_{k,m}^{k,\ell,m}=\binom{m_1}{m}\binom{m_2}{\ell}\binom{m_1-m}{k}\neq 0.$
For $a,b,c\in \NN_{0}$ we denote by $P_{a,b,c}$ the projection to the linear subspace $\Lin\{e_{1}^{i_1}e_2^be_3^a\otimes f_1^{i_2}f_2^{c}f_3^{i_3}\colon i_j\in \NN_0\}.$
Applying projections $P_{a,b,c}$ repeatedly in the lexicographic ordering of indices $(a,b,c)$ where $(b,c,a)\in S_d$ on \eqref{eq-lin-ind} and using \eqref{allveceq} we deduce inductively
that each $\alpha_{k,\ell,m}$ in \eqref{eq-lin-ind} is zero. 
Namely, first
	$$0=P_{d,0,0}\Big(\sum_{(k,\ell,m)\in S_d} \alpha_{k,\ell,m} v_{k,\ell,m}\Big)=\alpha_{0,0,d}\cdot \beta_{0,d}^{0,0,d} e_1^{m_1-d}e_3^d\otimes f_1^{m_2}$$
implies that $\alpha_{0,0,d}=0$ (since $\beta_{0,d}^{0,0,d}\neq 0$).
Now fix $(a_0,b_0,c_0)$ and assume that $\alpha_{b,c,a}=0$ for all $(a,b,c)\succ_{\mathrm{lex}} (a_0,b_0,c_0)$. Then	
	\begin{eqnarray*}
		0	&=&	P_{a_0,b_0,c_0}\Big(\sum_{(k,\ell,m)\in S_d} \alpha_{k,\ell,m} v_{k,\ell,m}\Big)\\
			&=&	\alpha_{b_0,c_0,a_0}\cdot \beta_{b_0,a_0}^{b_0,c_0,a_0} e_1^{m_1-b_0-a_0}e_2^{b_0}e_3^{a_0}\otimes f_1^{m_2-c_0}f_2^{c_0}
	\end{eqnarray*}
implies that $\alpha_{b_0,c_0,a_0}=0$ (since $\beta_{b_0,a_0}^{b_0,c_0,a_0}\neq 0$).
\end{proof}


\begin{lemma}\label{sl3-gen-on-v_klj-v2}
	For every $d,m_1,m_2 \in \NN_0$ with $m_1 \ge d$ and $m_2 \ge d$, generators of $\sl_3$ map vectors $v_{k,\ell,m}$, $k,\ell,m\in \NN_0$, 
	by the following rules:
	\begin{eqnarray}
		\pi_{m_1,m_2}(H_1) v_{k,\ell,m} &=& \alpha v_{k,\ell,m},							\label{sl3-gen-on-v_kl-eq7-v2}\\
		\pi_{m_1,m_2}(H_2) v_{k,\ell,m} &=& \beta   v_{k,\ell,m},							\label{sl3-gen-on-v_kl-eq8-v2}\\
		\pi_{m_1,m_2}(Y_1) v_{k,\ell,m} &=& v_{k+1,\ell,m},										\label{sl3-gen-on-v_kl-eq4-v2}\\
		\pi_{m_1,m_2}(Y_2) v_{k,\ell,m} &=& v_{k,\ell+1,m} + k v_{k-1,\ell,m+1},						\label{sl3-gen-on-v_kl-eq5-v2}\\
		\pi_{m_1,m_2}(Y_3) v_{k,\ell,m} &=& v_{k,\ell,m+1},										\label{sl3-gen-on-v_kl-eq6-v2}\\
		\pi_{m_1,m_2}(X_1) v_{k,\ell,m} &=& \gamma  v_{k-1,\ell,m}-mv_{k,\ell+1,m-1}					\label{sl3-gen-on-v_kl-eq1-v2}\\
		\pi_{m_1,m_2}(X_2) v_{k,\ell,m} &=& \delta v_{k,\ell-1,m}+m v_{k+1,\ell,m-1}.			\label{sl3-gen-on-v_kl-eq2-v2}\\
		\pi_{m_1,m_2}(X_3) v_{k,\ell,m} &=& \xi v_{k-1,\ell-1,m}+\zeta v_{k,\ell,m-1}				\label{sl3-gen-on-v_kl-eq3-v2},
	\end{eqnarray}
	where
	\begin{multicols}{2}
		\begin{equation*}
			\begin{aligned}
				\alpha	&=	(m_1 -2k + \ell-m),\\
				\gamma	&=	k(m_1-k+1+\ell-m),\\
				\xi		&=	-k\ell(m_2-\ell+1),
			\end{aligned}
		\end{equation*}
		\vfill
	\columnbreak
		\begin{equation*}
			\begin{aligned}		
				\beta		&=	(m_2 +k -2\ell-m),\\
				\delta	&=	\ell (m_2-\ell+1),\\
				\zeta		&=	m(m_1+m_2+1-\ell-k-m).
			\end{aligned}
		\end{equation*}
	\end{multicols}
\end{lemma}

\begin{proof}
	We write $\pi:=\pi_{m_1,m_2}$.
		Equalities \eqref{sl3-gen-on-v_kl-eq7-v2} and \eqref{sl3-gen-on-v_kl-eq8-v2} follow by the following facts:  
	\begin{itemize}
		\item $v_{0,0,0}$ is a weight vector corresponding to the weight $(m_1,m_2)$.
		\item $\pi(Y_1)$, $\pi(Y_2)$, $\pi(Y_3)$  are root vectors corresponding to 
			the roots $(-2,1)$, $(1,-2)$, $(-1,-1)$.
		\item A vector $v_{k,\ell,m}$ is nonzero by Proposition \ref{nonzero-proposition}.
	\end{itemize}
	 
	 The equality \eqref{sl3-gen-on-v_kl-eq4-v2} is clear while  \eqref{sl3-gen-on-v_kl-eq6-v2} follows by the fact that $Y_3$ commutes with 
	 $Y_1$ and $Y_2$ in $U(\sl_3)$.
	
	The remaining equalities can be proved by induction on lexicographically increasing triples $(k,\ell,m)$.
	For examples we will prove \eqref{sl3-gen-on-v_kl-eq5-v2} and \eqref{sl3-gen-on-v_kl-eq1-v2}.
	
	The base of induction $(k,\ell,m)=(0,0,0)$ for \eqref{sl3-gen-on-v_kl-eq5-v2} is established
	 by calculating $\pi(Y_2) v_{0,0,0} = v_{0,1,0}.$ Now fix a triple $(k_0,\ell_0,m_0)$ and assume that \eqref{sl3-gen-on-v_kl-eq5-v2} is true for every triple 
	 $(k,\ell,m)$ such that $(k_0,\ell_0,m_0)\succ_{\mathrm{lex}} (k,\ell,m)$. We separate two cases: \\
	
	 \noindent \textbf{Case 1:} $k_0>0$. By  the relation $Y_2Y_1=Y_1Y_2+Y_3$ from $U(\sl_3)$ and the fact that $Y_3$ commutes with $Y_1$ and $Y_2$ we have that
		$$\pi(Y_2) v_{k_0,\ell_0,m_0}	=	\pi(Y_1)\pi(Y_2) v_{k_0-1,\ell_0,m_0}+ v_{k_0-1,\ell_0,m_0+1},$$
	Now we use the induction hypothesis for $(k_0-1,\ell_0,m_0)$ and get
		$$\pi(Y_2) v_{k_0,\ell_0,m_0}	=	v_{k_0,\ell_0+1,m_0}+k_0v_{k_0-1,\ell_0,m_0+1}.$$	
		
	 \noindent \textbf{Case 2:} $k_0=0$.	We have $\pi(Y_2) v_{0,\ell_0,m_0}=v_{0,\ell_0+1,m_0}$ which is \eqref{sl3-gen-on-v_kl-eq5-v2}.\\
	 
	Now we prove \eqref{sl3-gen-on-v_kl-eq1-v2}. The base of induction $(k,\ell,m)=(0,0,0)$ is established
	 by calculating 
	 	$$\pi(X_1) v_{0,0,0} = \psi_1(X_1)e_1^{m_1} \otimes f_1^{m_2}+e_1^{m_1} \otimes \psi_1(X_1)f_1^{m_2}=0.$$
	 Now fix a triple $(k_0,\ell_0,m_0)$ and assume that \eqref{sl3-gen-on-v_kl-eq1-v2} is true for every triple 
	 $(k,\ell,m)$ such that $(k_0,\ell_0,m_0)\succ_{\mathrm{lex}} (k,\ell,m)$. We separate three cases: \\
	
	 \noindent \textbf{Case 1:} $k_0>0$. By the relation $X_1Y_1=Y_1X_1+H_1$ from $U(\sl_3)$ we have that
		$$\pi(X_1) v_{k_0,\ell_0,m_0}	=	\pi(Y_1)\pi(X_1) v_{k_0-1,\ell_0,m_0}+ \pi(H_1)v_{k_0-1,\ell_0,m_0},$$
	Now we use the induction hypothesis for $(k_0-1,\ell_0,m_0)$ for the first term, the equality \eqref{sl3-gen-on-v_kl-eq7-v2} for the second term and after a short calculation
	we get \eqref{sl3-gen-on-v_kl-eq1-v2}.\\
			
	 \noindent \textbf{Case 2:} $k_0=0$, $\ell_0>0$. By the relation $X_1Y_2=Y_2X_1$ from $U(\sl_3)$ we have that
		$$\pi(X_1) v_{0,\ell_0,m_0}	=	\pi(Y_2)\pi(X_1) v_{0,\ell_0-1,m_0},$$
	and by the induction hypothesis for $(0,\ell_0-1,m_0)$ we get \eqref{sl3-gen-on-v_kl-eq1-v2}.\\
	
	\noindent \textbf{Case 3:} $k_0=0$, $\ell_0=0$, $m_0>0$. By the relation $X_1Y_3=Y_3X_1-Y_2$ from $U(\sl_3)$ we have that
		$$\pi(X_1) v_{0,0,m_0}	=	\pi(Y_3)\pi(X_1) v_{0,0,m_0-1}-v_{0,1,m_0-1},$$
	and by the induction hypothesis for $(0,0,m_0-1)$ we get \eqref{sl3-gen-on-v_kl-eq1-v2}.
	\end{proof}

\begin{proposition}\label{propklm}
	For every $d,m_1,m_2 \in \NN_0$ with $m_1,m_2$ big enough, the vectors 
		\begin{equation}\label{eqpropklm}
			\pi_{m_1,m_2}(Y_1^{j_1} Y_2^{j_2} Y_3^{j_3} X_1^{\ell_1}X_2^{\ell_2}X_3^{\ell_3} H_1^{r_1} H_2^{r_2})
				\Big(\sum_{t=1}^{L} v_{k(t),\ell(t),m(t)} \Big),
		\end{equation}
	are linearly independent,	
	where the powers $j_1,j_2,j_3,\ell_1,\ell_2,\ell_3,r_1,r_2 \in \NN_0$ are such that $\displaystyle \sum_{i=1}^3 j_i+\sum_{i=1}^3 \ell_i+\sum_{i=1}^2 r_i \le d$, $j_2 \ell_2=0$, $r_2 \le 2$ and
	the indices $k(t),\ell(t),m(t)$ for $t=1,\ldots, L$, with $L:=4d^3+4d^2+2d+1$, are defined by
	$$k(t)=(3d+1)t, \quad \ell(t)=t^{2d+1}, \quad m(t)=t^{4d^2+2d+1}.$$
\end{proposition}

\begin{proof}
We write $\vec{Y}:=(Y_1,Y_2,Y_3)$, $\vec{X}:=(X_1,X_2,X_3)$, $\vec{H}:=(H_1,H_2)$, $\vec{j}:=(j_1,j_2,j_3)$, $\vec{\ell}:=(\ell_1,\ell_2,\ell_3)$, $\vec{r}:=(r_1,r_2)$ and
	$$\vec{Y}^{\vec{j}}\vec{X}^{\vec{\ell}}\vec{H}^{\vec{r}}:=Y_1^{j_1} Y_2^{j_2} Y_3^{j_3} X_1^{\ell_1}X_2^{\ell_2}X_3^{\ell_3}H_1^{r_1} H_2^{r_2}.$$
Lemma \ref{sl3-gen-on-v_klj-v2} implies that 
	\begin{equation}\label{maineqpropklm}
	\begin{split}
		\pi_{m_1,m_2}(\vec{Y}^{\vec{j}}\vec{X}^{\vec{\ell}}\vec{H}^{\vec{r}})v_{k,\ell,m} 
		&=\sum_{s=0}^{j_2+\ell_1+\ell_2+\ell_3} 
			c_{\vec{j},\vec{\ell},\vec{r},s}(k,\ell,m)\cdot\\
		&	v_{k-\ell_3-\ell_1-j_2+j_1+s,\ell-\ell_3-\ell_2+s,m+j_2+j_3-s}
	\end{split}
	\end{equation}
where 
$c_{\vec{j},\vec{\ell},\vec{r},s}(k,\ell,m)$ are polynomials in $k,\ell,m$. 
Let $S$ be the endomorphism of $V_{\pi_{m_1,m_2}}$ defined by
	\begin{equation}
	\label{def-s}
		S(v_{k,\ell,m})=\left\{ \begin{array}{cc} v_{k+1,\ell+1,m-1} & \text{ if } m \ge 1, \\ 0 & \text{ if } m=0. \end{array} \right.
	\end{equation}
Consider the operator 
	\begin{equation*}
		C_{\vec{j},\vec{\ell},\vec{r}}(k,\ell,m,S):=\sum_{s=0}^{j_2+\ell_1+\ell_2+\ell_3} c_{\vec{j},\vec{\ell},\vec{r},s}(k,\ell,m) S^s
	\end{equation*}
The equation \eqref{maineqpropklm} can now be rewritten as
	\begin{equation}\label{maineqpropklmS}
		\pi_{m_1,m_2}(\vec{Y}^{\vec{j}}\vec{X}^{\vec{\ell}}\vec{H}^{\vec{r}})v_{k,\ell,m} =C_{\vec{j},\vec{\ell},\vec{r}}(k,\ell,m,S)v_{k-\ell_3-\ell_1-j_2+j_1,\ell-\ell_3-\ell_2,m+j_2+j_3}
	\end{equation}
To compute the leading term of
	$C_{\vec{j},\vec{\ell},\vec{r}}(k,\ell,m,S)$ with respect to a monomial ordering $\succ$ defined below, we first introduce new variables
	\begin{equation}\label{new-var}
		x:=k-2\ell-m, \quad y:=-k+\ell-m, \quad z:=-2k+\ell-m.
	\end{equation}
Note that we have that
	\begin{equation}\label{old-var}
		k= y-z, \quad \ell=\frac{1}{3} (-x+3 y-2 z), \quad m= \frac{1}{3} (-x-3 y+z).
	\end{equation}		
Now consider the lexicographic ordering induced by 
	\begin{equation}\label{mon-ord}
		x \succ y \succ z \succ S. 
	\end{equation}
Using Lemma \ref{sl3-gen-on-v_klj-v2}  we see that the leading term of 
$C_{\vec{j},\vec{\ell},\vec{r}}(k,\ell,m,S)$ is the same as the leading term of
	$$	(k+S)^{j_2}(k(\ell-k-m)-m S)^{\ell_1} (-\ell^2+m S)^{\ell_2} (k \ell^2-m(k+\ell+m)S)^{\ell_3} \cdot$$
	$$	\cdot (-2k+\ell-m)^{r_1}(-2\ell+k-m)^{r_2}$$
which is equal to 
\begin{gather*}
	y^{j_2} \Big(\frac{S x}{3}\Big)^{\ell_1}\Big(-\frac{x^2}{9}\Big)^{\ell_2}\Big(\frac{x^2 y}{9}\Big)^{\ell_3} z^{r_1} x^{r_2} =
	\frac{(-1)^{\ell_2}}{3^{\ell_1+2(\ell_2+\ell_3)}} x^{\ell_1+2(\ell_2+\ell_3)+r_2} y^{j_2+\ell_3} z^{r_1} S^{\ell_1}.
\end{gather*}

We denote by $\Gamma_d$ the set of all tuples $(\vec{j},\vec{\ell},\vec{r})$ satisfying
	$$\sum_{i=1}^3 j_i+\sum_{i=1}^3 \ell_i+\sum_{i=1}^2 r_i \le d,\quad j_2\ell_2=0\quad\text{and}\quad r_2\leq 2.$$
Assume that
\begin{equation} \label{ldpropsumklm}
	\sum_{(\vec{j},\vec{\ell},\vec{r})\in \Gamma_d} 
	\lambda_{\vec{j},\vec{\ell},\vec{r}}\cdot
	\pi_{m_1,m_2}(\vec{Y}^{\vec{j}}\vec{X}^{\vec{\ell}}\vec{H}^{\vec{r}})
				\Big(\sum_{t=1}^{L} v_{k(t),\ell(t),m(t)} \Big)=0
\end{equation}
By the choice of $k(t),\ell(t),m(t)$ we have that triples of $v$-indices appearing in 
	$$\pi_{m_1,m_2}(\vec{Y}^{\vec{j}}\vec{X}^{\vec{\ell}}\vec{H}^{\vec{r}}) v_{k(t),\ell(t),m(t)}$$
are always different from triples of $v$-indices appearing in 
	$$\pi_{m_1,m_2}(\vec{Y}^{\vec{j}}\vec{X}^{\vec{\ell}}\vec{H}^{\vec{r}}) v_{k(t'),\ell(t'),m(t')}$$
	if $t \ne t'$. Therefore, the equation \eqref{ldpropsumklm} implies that for every $t=1,\ldots,L$, we have that
	\begin{equation*}
		\sum_{(\vec{j},\vec{\ell},\vec{r})\in \Gamma_d} 
		\lambda_{\vec{j},\vec{\ell},\vec{r}}\cdot
		\pi_{m_1,m_2}(\vec{Y}^{\vec{j}}\vec{X}^{\vec{\ell}}\vec{H}^{\vec{r}})
		 v_{k(t),\ell(t),m(t)}=0.
	\end{equation*}
The equation \eqref{maineqpropklmS} implies that
	\begin{equation}\label{maineqpropklm2} 
		0 = 	\sum_{(\vec{j},\vec{\ell},\vec{r})\in \Gamma_d} \lambda_{\vec{j},\vec{\ell},\vec{r}}\cdot C_{\vec{j},\vec{\ell},\vec{r}}(k(t),\ell(t),m(t),S)v_{k(t)-\ell_3-\ell_1-j_2+j_1,\ell(t)-\ell_3-\ell_2,m(t)+j_2+j_3}
	\end{equation}
For every $(\vec{j},\vec{\ell},\vec{r})\in \Gamma_d$ let us define the set
 	\begin{eqnarray*}
		\Delta_{(\vec{j},\vec{\ell},\vec{r})}
		&:=&\{ (d_1,d_2,d_3)\in \ZZ^3\colon  d_1=-\ell_3-\ell_1-j_2+j_1+s, d_2=-\ell_3-\ell_2+s,\\
		&&	d_3=j_2+j_3-s \;\;\text{for some}\;\; 0 \le s \le j_2+\sum_{i=1}^3\ell_i\}.
	\end{eqnarray*}	
Fix a vector $\vec{e}:=(e_1,e_2)\in \ZZ^2$ and define a set
\begin{eqnarray*}
\Lambda_{\vec{e}} & := & \{ (\vec{j},\vec{\ell},\vec{r})\in \Gamma_d\colon (e_1+d_2,d_2,e_2-d_2)\in \Delta_{(\vec{j},\vec{\ell},\vec{r})}\;\text{for some}\; d_2\in \ZZ\} \\
& = & \{ (\vec{j},\vec{\ell},\vec{r})\in \Gamma_d\colon  j_1=j_2+\ell_1-\ell_2+e_1, j_3= -j_2+\ell_2+\ell_3+e_2\} \\
\end{eqnarray*}
Note that sets $\Lambda_{\vec{e}}$ are pairwise disjoint and that they cover $\Gamma_d$. Let us define a vector function $\vec{f}$ of $j_2,\vec{\ell},\vec{e}$ by
$$\vec{f}(j_2,\vec{\ell},\vec{e})=(j_2+\ell_1-\ell_2+e_1,j_2,-j_2+\ell_2+\ell_3+e_2).$$
Clearly, $\Lambda_{\vec{e}} = \{ (\vec{j},\vec{\ell},\vec{r})\in \Gamma_d\colon  (j_1,j_2,j_3)=\vec{f}(j_2,\vec{\ell},\vec{e})\}$.
Let $\Lambda_{\vec{e}}'$ be the projection of $\Lambda_{\vec{e}}$ along $j_1$ and $j_3$.
The equation \eqref{maineqpropklm2} implies that
	\begin{eqnarray*}
		0
		&=&	\sum_{(j_2,\vec{\ell},\vec{r})\in \Lambda'_{\vec{e}}}
			\lambda_{\vec{f}(j_2,\vec{\ell},\vec{e}),\vec{\ell},\vec{r}} \cdot
			C_{\vec{f}(j_2,\vec{\ell},\vec{e}),\vec{\ell},\vec{r}}(k(t),\ell(t),m(t),S)\\ 
		&&    \hspace{2cm} v_{k(t)+e_1-\ell_2-\ell_3,\ell(t)-\ell_2-\ell_3,m(t)+e_2+\ell_2+\ell_3},\\
		&=&	\sum_{(j_2,\vec{\ell},\vec{r})\in \Lambda'_{\vec{e}}}
			\left(\lambda_{\vec{f}(j_2,\vec{\ell},\vec{e}),\vec{\ell},\vec{r}} \cdot
			C_{\vec{f}(j_2,\vec{\ell},\vec{e}),\vec{\ell},\vec{r}}(k(t),\ell(t),m(t),S)\right.\\ 
		&&    \hspace{2cm} \left.  S^{d-\ell_2-\ell_3} \right) v_{k(t)+e_1-d,\ell(t)-d,m(t)+e_2+d}.
	\end{eqnarray*}
Defining operators
	$$P_{j_2,\vec{\ell},\vec{r}}:=C_{\vec{f}(j_2,\vec{\ell},\vec{e}),\vec{\ell},\vec{r}}(k(t),\ell(t),m(t),S)S^{d-\ell_2-\ell_3},$$
we get that
\begin{equation}\label{zero-linear-comb}
0=\sum_{(j_2,\vec{\ell},\vec{r})\in \Lambda'_{\vec{e}}}
		\left(\lambda_{\vec{f}(j_2,\vec{\ell},\vec{e}),\vec{\ell},\vec{r}} P_{j_2,\vec{\ell},\vec{r}}\right)
		 v_{k(t)+e_1-d,\ell(t)-d,m(t)+e_2+d}.
\end{equation}		 
We will prove by contradiction that $\lambda_{\vec{f}(j_2,\vec{\ell},\vec{e}),\vec{\ell},\vec{r}}=0$ for all $j_2,\vec{\ell},\vec{r}$ 
and hence	 $\lambda_{\vec{j},\vec{\ell},\vec{r}}=0$ for all $\vec{j},\vec{\ell},\vec{r}\in \Gamma_d$ in \eqref{ldpropsumklm}.
Among tuples $(j_2,\vec{\ell},\vec{r})$ with $\lambda_{\vec{f}(j_2,\vec{\ell},\vec{e}),\vec{\ell},\vec{r}}\neq 0$ choose a tuple $(j_2',\vec{\ell}',\vec{r}')$
such that the operator $P_{j_2,\vec{\ell},\vec{r}}$ has the highest leading term with respect to the monomial ordering \eqref{mon-ord}.
By the following claim such tuple is unique.\\

\noindent \textbf{Claim 1:} Different operators $P_{j_2,\vec{\ell},\vec{r}}$ have different leading terms.\\ 

From the discussion above, it follows that the leading term of the operator
	$P_{j_2,\vec{\ell},\vec{r}}$ is 
\begin{gather*}
	\frac{(-1)^{\ell_2}}{3^{\ell_1+2(\ell_2+\ell_3)}} {x(t)}^{\ell_1+2(\ell_2+\ell_3)+r_2} {y(t)}^{j_2+\ell_3} {z(t)}^{r_1} S^{\ell_1+d-\ell_2-\ell_3}.
\end{gather*}
Pick any $\alpha,\beta,\gamma,\delta \in \NN_0$. We will show that there exists at most one tuple $(j_2,\vec{\ell},\vec{r})\in\NN_0^6$ such that 
	\begin{gather}
		\label{ueq1} \ell_1+2(\ell_2+\ell_3)+r_2 = \alpha \\
		\label{ueq2} j_2+\ell_3 = \beta \\
		\label{ueq3} r_1=\gamma \\
		\label{ueq4} \ell_1+d-\ell_2-\ell_3=\delta \\
		\label{ueq5} j_2 \ell_2=0 \\
		\label{ueq6} r_2 \le 2
	\end{gather}
Subtracting \eqref{ueq4} from \eqref{ueq1} we obtain
	\begin{equation}\label{div3eq}
		3(\ell_2+\ell_3)+r_2=\alpha-\delta+d
	\end{equation}
which together with \eqref{ueq6} implies that
	\begin{gather}
		\label{ueq7} \ell_2+\ell_3=(\alpha-\delta+d) \operatorname{div} 3 =: \varepsilon\\
		r_2 = (\alpha-\delta+d) \operatorname{mod} 3 
	\end{gather}
Equations \eqref{ueq4} and \eqref{ueq7} imply that 
\begin{equation}
\ell_1=\delta+\varepsilon-d.
\end{equation}
Subtracting \eqref{ueq7} from \eqref{ueq2} we obtain
\begin{equation}
j_2-\ell_2=\beta-\varepsilon
\end{equation}
which together with \eqref{ueq5} implies that
\begin{gather}
j_2=(\beta-\varepsilon)^+ :=\max\{\beta-\varepsilon,0\} \\
\label{ueq8} \ell_2=(\beta-\varepsilon)^- :=\max\{\varepsilon-\beta,0\} 
\end{gather}
From \eqref{ueq7} and \eqref{ueq8} we obtain
\begin{equation}
\ell_3=\varepsilon-(\beta-\varepsilon)^-
\end{equation}
We already know that $r_1=\gamma$ from \eqref{ueq3}. 
This proves Claim 1.\\

For the tuple $(j_2',\vec{\ell}',\vec{r}')$ let $\alpha',\beta',\gamma',\delta'$ be defined as in \eqref{ueq1}-\eqref{ueq4}.
Now we observe the coefficients at the vector
	$$v_{k(t)+e_1-d+\delta',\ell(t)-d+\delta',m(t)+e_2+d-\delta'}$$
on both sides of \eqref{zero-linear-comb} and get
	$$0=\frac{(-1)^{\ell'_2}}{3^{\ell'_1+2(\ell'_2+\ell'_3)}} {x(t)}^{\alpha'} {y(t)}^{\beta'} {z(t)}^{\gamma'}+
		\sum_{\substack{\alpha,\beta,\gamma\in \NN_0,\\
			0\leq \alpha+\beta+\gamma\leq 2d,\\
			(\alpha',\beta',\gamma')\succ (\alpha,\beta,\gamma)}} c_{\alpha,\beta,\gamma} x(t)^{\alpha}y(t)^{\beta}z(t)^{\gamma},$$
for some $c_{\alpha,\beta,\gamma}\in \CC$. Since this must hold for all $t=1,\ldots,L$, this is a contradiction by the following claim.\\

\noindent \textbf{Claim 2:} All vectors 
		$$\begin{pmatrix} x(t)^{\alpha_1} y(t)^{\alpha_2} z(t)^{\alpha_3} \end{pmatrix}_{t=1,\ldots,L}$$
	where $0 \le \sum_{i=1}^3 \alpha_i \le 2d$ are linearly independent.\\

By Vandermonde determinant one can show that all vectors 
		$$\begin{pmatrix} k(t)^{\alpha_1} \ell(t)^{\alpha_2} m(t)^{\alpha_3} \end{pmatrix}_{t=1,\ldots,L}=$$
		$$\begin{pmatrix} (3d+1)^{\alpha_1} \cdot t^{\alpha_1+\alpha_2\cdot (2d+1)+\alpha_3\cdot (4d^2+4d+1)} \end{pmatrix}_{t=1,\ldots,L}$$
where $0 \le \sum_{i=1}^3 \alpha_i \le 2d$ are linearly independent. Indeed, for different triples $(\alpha_1,\alpha_2,\alpha_3)$ satisfying 
$0 \le \sum_{i=1}^3 \alpha_i \le 2d$, the exponents
		$$\alpha_1+\alpha_2\cdot (2d+1)+\alpha_3\cdot (4d^2+4d+1)$$
are different, with the highest exponent $L-1$ reached at $\alpha_1=\alpha_2=0$, $\alpha_3=2d$.
By using \eqref{new-var} and \eqref{old-var} we see that
	$$\span\{\begin{pmatrix} x(t)^{\alpha_i} y(t)^{\alpha_i} z(t)^{\alpha_i}  \end{pmatrix}_{t=1,\ldots,L}\colon 0 \le \alpha_1+\alpha_2+\alpha_3\le 2d\}$$
is equal to 
	$$\span\{\begin{pmatrix} k(t)^{\alpha_i} \ell(t)^{\alpha_i} m(t)^{\alpha_i} \end{pmatrix}_{t=1,\ldots,L}\colon 0 \le \alpha_1+\alpha_2+\alpha_3\le 2d\}.$$
Therefore also all vectors 
	$$\begin{pmatrix} x(t)^{\alpha_1} y(t)^{\alpha_2} z(t)^{\alpha_3} \end{pmatrix}_{t=1,\ldots,L}$$
where $0 \le \sum_{i=1}^3 \alpha_i \le 2d$ are linearly independent. This proves Claim 2.
\end{proof}

\subsection{An explicit basis over the center of $U(\sl_3)$}	
\label{usl3-gen}

It is well-known that the center of $U(\sl_3)$ is generated by two algebraically independent elements $Z_2$ and $Z_3$ which are also called Casimir operators.
The algorithm for computing $Z_2$ and $Z_3$ can be found in \cite{ga}, while explicit expressions are in \cite[p.\ 984]{cat}. We have
\begin{eqnarray*}
	Z_2	&=& H_1^2+H_1H_2+H_2^2+3Y_1X_1+3Y_2X_2+3Y_3X_3+3H_1+3H_2
\end{eqnarray*}
and
\begin{equation*}\label{defz3}
\begin{split}
	Z_3
		&=3 Y_1 Y_2 X_3 + 3 Y_3 X_1 X_2 +\frac{1}{9}(H_1 + 2 H_2)(6 + 2 H_1 + H_2)(-3 + H_1 - H_2)  \\ 
		&+ Y_1 X_1 (H_1+2H_2)-Y_2 X_2(6+2 H_1+H_2)+Y_3 X_3 (-3+H_1-H_2)
\end{split} 
\end{equation*}
 (Our choice of $Z_2$ and $Z_3$ is equal to $h$ and $-\frac{1}{9}k-h$ in the notation of \cite[p.\ 984]{cat}.)

By \cite[Schur's Lemma]{hum} an irreducible representation maps a central element into a scalar multiple of identity. 
Therefore it is enough to calculate $\pi_{m_1,m_2}(Z_i)v_{0,0,0}$ to determine this scalar.
From Lemma \ref{sl3-gen-on-v_klj-v2}, we get that
\begin{equation}\label{defcasinv}
\begin{split}
d_2(m_1,m_2) &:=	\pi_{m_1,m_2}(Z_2)=m_1^2+m_1m_2+m_2^2+3m_1+3m_2,\\
d_3(m_1,m_2) &:= \pi_{m_1,m_2}(Z_3)=\frac{1}{9}(m_1 + 2 m_2)(6 + 2 m_1 + m_2) (-3 + m_1 - m_2).
\end{split}
\end{equation}

\begin{proposition}\label{explicitbasis}
Monomials 
\begin{equation}\label{genmon}
	Y_1^{j_1} Y_2^{j_2} Y_3^{j_3} X_1^{\ell_1}X_2^{\ell_2}X_3^{\ell_3} H_1^{r_1} H_2^{r_2}
\end{equation}
where the powers $j_1,j_2,j_3,\ell_1,\ell_2,\ell_3,r_1,r_2 \in \NN_0$ are such that $j_2 \ell_2=0$ and $r_2 \le 2$
form a basis of $U(\sl_3)$ over its center.
\end{proposition}

\begin{proof}
Linear independence of monomials \eqref{genmon} follows from Proposition \ref{propklm}. It remains to prove that they
span $U(\sl_3)$ over its center.

Let $U(\sl_3)_k$ denote the $\CC$-linear span of monomials of the form
\begin{equation}\label{degrevlex0}
	Y_2^{\ell_2} X_2^{j_2} H_2^{r_2} Y_3^{\ell_3}X_3^{j_3} Y_1^{\ell_1} X_1^{j_1}   H_1^{r_1}
\end{equation}
of degree at most $k$ where the degree equals to the sum of the exponents. We write $\text{deg}(m)$ for the degree of the
monomial of the form \eqref{degrevlex0}.
We define the set
	$$M_k:=\{ m\text{ of the form } \eqref{degrevlex0} \colon \text{deg}(m)\leq k,\;j_2\ell_2=0\text{ and } r_2\leq 2 \}.$$
We will prove that 
	\begin{equation}\label{equality}
		U(\sl_3)_k=\span_Z(M_k)
	\end{equation}
where $Z$ stands for the center of $U(\sl_3)$. 
It suffices to prove that every monomial of the form \eqref{degrevlex0} belongs to $\span_Z(M_k)$.
Let us order the monomials \eqref{degrevlex0} with respect to the degree reverse lexicographic ordering.
Note that the largest monomial in the definition of $Z_2$ is $3 Y_2 X_2$  and that the largest monomial
in the definition of $Z_3+\frac{1}{3}Z_2(6+2H_1+H_2)$ is $\frac{1}{9} H_2^3$. If we express $Y_2 X_2$ by $Z_2$ and other monomials
and similarly $H_2^3$ by $Z_3+\frac{1}{3}Z_2(6+2H_1+H_2)$ and other monomials we get two substitution rules. (Note that the first substitution
rule decreases $\min\{j_2,\ell_2\}$ but it can increase $r_2$ and that the second substitution rule decreases $r_2$
but can increase $\min\{j_2,\ell_2\}$.)
%
If we start with a monomial with either $j_2 \ell_2>0$ or $r_2 \ge 3$ and keep applying these substitution rules
whenever possible we get a decreasing sequence of expressions with respect to the degree reverse lexicographic ordering.
Since this ordering is known to be a well-ordering, this sequence must stop at some point.
This means that we finish with an expression whose monomials all satisfy $j_2 \ell_2=0$ and $r_2 \le 2$. This proves \eqref{equality}.

By the PBW theorem we know that every element of $U(\sl_3)$ belongs to $U(\sl_3)_k$ for some $k\in \NN_0$.
We define the set
	$$\widetilde{M}_k:=\{ m\text{ of the form } \eqref{genmon} \colon \text{deg}(m)\leq k,\;j_2\ell_2=0\text{ and } r_2\leq 2 \},$$
where $\deg(m)$ is a sum of exponents in $m$.
To finish the proof of the proposition it remains to prove that 
	\begin{equation}\label{rem-eq}
		\span_Z M_k=\span_Z\widetilde{M}_k.
	\end{equation}
We prove \eqref{rem-eq} by induction on $k$.
The base of induction $k=1$ is clear. We assume that \eqref{rem-eq} for all $k\leq n$ for some $n\in \NN$.
By the relations in $U(\sl_3)$ we have that
	\begin{equation*}
		Y_2^{\ell_2} X_2^{j_2} H_2^{r_2} Y_3^{\ell_3}X_3^{j_3} Y_1^{\ell_1} X_1^{j_1}   H_1^{r_1}=
		Y_1^{\ell_1}Y_2^{\ell_2}Y_3^{\ell_3}X_1^{\ell_1}X_2^{\ell_2}X_3^{\ell_3} H_1^{r_1}  H_2^{r_2}+m'
	\end{equation*}
where $m'$ is a $\ZZ$-linear combination of monomials of the form \eqref{degrevlex0} of degree at most 
	$n-1:=\sum_{i=1}^3 (\ell_i+j_i)+r_1+r_2-1.$
By \eqref{equality} we have that $m'\in \span_Z M_{n-1}$ and by the induction hypothesis we have that $m'\in \span_Z \widetilde{M}_{n-1}$.
This proves \eqref{rem-eq}.
\end{proof}

\begin{lemma}\label{technical-lemma2}
\begin{enumerate}
\item\label{tech-lemma-pt1} Every polynomial $g \in \CC[x,y]$ which satisfies $$g(m_1,m_2)=0$$ for all sufficiently large integers $m_1,m_2$ is equal to zero.
\item\label{tech-lemma-pt2} Every polynomial $f \in \CC[x,y]$ which satisfies 
\begin{equation}\label{algind}
	f(d_2(m_1,m_2),d_3(m_1,m_2))=0
\end{equation}
for all sufficiently large integers $m_1,m_2$ is equal to zero.
\item\label{tech-lemma-pt3} Every vectors $u_1,\ldots,u_k \in \CC[x,y]^n$, such that the vectors
\begin{equation}\label{zlindep}
	u_i(d_2(m_1,m_2),d_3(m_1,m_2)), i=1,\ldots,k
\end{equation}
are linearly dependent over $\CC$ for all sufficiently large integers $m_1,m_2$,
are linearly dependent over $\CC[x,y]$.
\end{enumerate}
\end{lemma}

\begin{proof}
Part \eqref{tech-lemma-pt1} is well-known and easy to prove.

To prove part \eqref{tech-lemma-pt2}, assume that \eqref{algind} is true for all sufficiently large integers $m_1,m_2$.
By part \eqref{tech-lemma-pt1}, it follows that \eqref{algind} is true for all $m_1,m_2 \in \CC$. 
Let us compute the partial derivatives of \eqref{algind} with respect to $m_1$ and $m_2$ by using the chain rule.
We get that
\begin{equation}
\left[\frac{\partial f}{\partial x}(d_2,d_3),\frac{\partial f}{\partial y}(d_2,d_3)\right]
\left[ \begin{array}{cc} \frac{\partial d_2}{\partial m_1} & \frac{\partial d_2}{\partial m_2} \\
\frac{\partial d_3}{\partial m_1} & \frac{\partial d_3}{\partial m_2} \end{array} \right] =
\left[0,0\right]
\end{equation}
Since
\begin{equation}
\det \left[ \begin{array}{cc} \frac{\partial d_2}{\partial m_1} & \frac{\partial d_2}{\partial m_2} \\
\frac{\partial d_3}{\partial m_1} & \frac{\partial d_3}{\partial m_2} \end{array} \right] =
	-3(1+m_1)(1+m_2)(2+m_1+m_2)
\end{equation}
is nonzero for all positive reals $m_1$ and $m_2$ it follows that
\begin{gather}
\frac{\partial f}{\partial x}(d_2(m_1,m_2),d_3(m_1,m_2))=0 \\
\frac{\partial f}{\partial y}(d_2(m_1,m_2),d_3(m_1,m_2))=0
\end{gather}
for all positive reals $m_1$ and $m_2$, thus for all complex $m_1,m_2$ by part (1).
By induction, we show that 
\begin{equation}
\frac{\partial^{i+j} f}{\partial^i x \, \partial^j y}(d_2(m_1,m_2),d_3(m_1,m_2))=0
\end{equation}
for all $i,j \in \NN_0$ and all $m_1,m_2 \in \CC$. It follows that $f \equiv 0$.

To prove part \eqref{tech-lemma-pt3}, consider the matrix $U=\left[u_1,\ldots,u_k \right]$.
By assumption, each maximal subdeterminant of $U(d_2(m_1,m_2),d_3(m_1,m_2))$
is zero for all sufficiently large integers $m_1$ and $m_2$. By part \eqref{tech-lemma-pt2},
it follows that each maximal subdeterminant of $U$ is identically zero.
Thus $u_1,\ldots,u_k$ are linearly dependent over $\CC(x,y)$.
By clearing denominators, we see that they are also linearly dependent over $\CC[x,y]$.
\end{proof}

We are now ready to prove Theorem \ref{mainthm3}.

%

\begin{proof}[Proof of Theorem \ref{mainthm3}]
	To prove the implication $\eqref{mainthm3-pt1} \Rightarrow \eqref{mainthm3-pt2}$, one
	has to show that there exists $d\in \NN$
	such that for every $m_1, m_2\in \NN$, $m_1\geq d$, $m_2\geq d$, the elements 
	$\pi_{m_1,m_2}(p_1)$,$\ldots$, $\pi_{m_1,m_2}(p_k)$ are linearly dependent, where $\pi_{m_1,m_2} \in \I_d$.
	Dividing the equation
	$\sum_{i=1}^k z_ip_i=0$ by the common factor of $z_1,\ldots,z_k$, where $z_i\in \CC[Z_2,Z_3]$ for each $i$, we may assume that
	$z_1,\ldots,z_k$ do not share a common nontrivial factor.
	Applying $\pi_{m_1,m_2}$ to the equation $\sum_{i=1}^k z_ip_i=0$, one gets
		$$0=\sum_{i=1}^k z_i(d_2(m_1,m_2),d_3(m_1,m_2))\pi_{m_1,m_2}(p_i).$$
	It suffices to prove that there exists $d\in \NN$ such that for every $m_1\geq d,m_2\geq d$ at least one of the coefficients
	$z_i(d_2(m_1,m_2),d_3(m_1,m_2))$ is nonzero. Let us assume on the contrary that such $d$ does not exist. Then there exists a sequence
	$(m_1^{(n)},m_2^{(n)})\in \NN^2$, $n\in \NN$, satisfying $\max\{m_1^{(n)},m_2^{(n)}\}< \min\{m_1^{(n+1)},m_2^{(n+1)}\}$ for every $n\in \NN$ and
	$z_i(d_2(m^{(n)}_1,m^{(n)}_2),d_3(m^{(n)}_1,m^{(n)}_2))=0$ for each $i$ and every $n\in \NN$.
	By the form of $d_2(m_1,m_2)$, the sequence $d_2(m^{(n)}_1,m^{(n)}_2)$ is strictly increasing 
	and hence the polynomials $z_1,\ldots,z_k$ share infinitely many common zeroes.
	This implies by Bezout's theorem (see \cite{ful}) that they share a nontrivial factor, leading to a contradiciton.

	The implication $\eqref{mainthm3-pt2} \Rightarrow \eqref{mainthm3-pt3}$ is trivial.

	The proof of the implication $\eqref{mainthm3-pt3} \Rightarrow \eqref{mainthm3-pt1}$
	is almost the same as the proof of the same implication of Theorem \ref{mainthm2}.
	Namely, the form of monomials $f_i$ is given by Proposition \ref{explicitbasis}, 
	the coefficients $t_i$ belong to $\CC[Z_2,Z_3]$, while Proposition \ref{irr-repr-prop} and Lemma \ref{technical-lemma}
	are replaced by Proposition \ref{propklm} and part \eqref{tech-lemma-pt3} of Lemma \ref{technical-lemma2}, respectively.
\end{proof}

\end{document}